\documentclass[12pt,reqno,oneside]{amsart}
\usepackage{amsmath, amssymb,latexsym}
\usepackage[mathscr]{eucal}
\usepackage{hyperref}

\numberwithin{equation}{section} 
\newtheorem{theorem}{Theorem}[section]
\newtheorem{corollary}[theorem]{Corollary}

\theoremstyle{definition}
\newtheorem{definition}[theorem]{Definition}

\theoremstyle{remark}
\newtheorem{remark}[theorem]{Remark}
\numberwithin{equation}{section}

\thanks{* Corresponding author, Email:gmsmoorthy@yahoo.com}
\begin{document}
\title [Coefficient bounds for a subclass ]{Coefficient Bounds For a subclass
Of \\bi-prestarlike functions \\Associated with the Chebyshev Polynomials  }
\author{Hatun \"{O}zlem G\"{U}NEY$^{1}$, G. Murugusundaramoorthy$^{2,*}$ K. Vijaya $^{2}$ and K.Thilagavathi $^{2}$}
\maketitle
\begin{center}
Faculty of Science, Department of Mathematics,\\ Dicle University, \\TR-21280 Diyarbak{\i}r, Turkey\\
 {\bf E-mail:{\it ozlemg@dicle.edu.tr}}\end{center} \vskip 6pt \begin{center}$^{*}$ Corresponding Author  \\$^{2}$School of Advanced Sciences,\\ VIT,  Vellore -
632014, India.\\ {\bf E-mail:~~{\it kvijaya@vit.ac.in,kthilagavathi@vit.ac.in}}
\end{center}

\begin{abstract}
In this paper, we introduce and investigate a new subclass of bi-prestarlike functions  defined in the open unit disk, associated with Chebyshev Polynomials. Furthermore, we find estimates of first two coefficients of functions in these classes, making use of the Chebyshev polynomials. Also, we obtain the Fekete-Szeg\"{o} inequalities for function in these classes. Several consequences of the results are also pointed out as corollaries.
\end{abstract}
\maketitle

\section{Introduction} Let $\mathcal{A}$ denote the class of analytic functions of the form
\begin{equation}\label{c7e1}
f(z)=z+\sum\limits_{n=2}^{\infty}a_n~z^n
\end{equation}
normalized by the conditions $f(0) = 0 = f'(0) - 1$ defined in the open unit disk $$\triangle = \{z \in \mathbb{C} : |z| < 1  \}.$$ Let $\mathcal{S}$ be the subclass of
$\mathcal{A}$ consisting of functions of the form
\eqref{c7e1} which are also univalent in $\triangle.$
Let $\mathcal{S}^{*}(\alpha)$ and $\mathcal{K}(\alpha)$ denote
the well-known subclasses of  $\mathcal{S},$
consisting of starlike and convex functions of order $\alpha,\ 0 \leq \alpha < 1,$ respectively.
The function
\begin{equation*}
s(z)=\frac{z}{(1-z)^{2(1 - \alpha) }}=z + \sum\limits_{n = 2}^\infty \Psi_n(\alpha) z^n
\end{equation*}where
\begin{equation}\label{psin}
    \Psi_n(\alpha)={\left( {\frac{{\prod\limits_{k = 2}^n
 {\left( {k - 2\alpha } \right)} }}{{\left( {n - 1} \right)!}}} \right)}
\end{equation}is the well-known extremal function for the class ${\mathcal{S}^*}\left( \alpha  \right).$
Also $f\in \mathcal{A}$ is said to be prestarlike functions of order $\alpha (0 \leq \alpha < 1),$ denoted by ${\mathcal{R}}\left( \alpha  \right)$ if
$f\ast s(z)\in {\mathcal{S}^*}\left( \alpha  \right).$ We note that $\mathcal{R}(1/2) = \mathcal{S}^*(1/2)$ and $\mathcal{R}(0) = \mathcal{K}(0).$ Using the convolution techniques, Ruscheweyh \cite{rush} introduced and studied
the class of prestarlike functions of order $\alpha$. For functions $f \in \mathcal{S},$ we have
$f \in \mathcal{K}(0) \Longleftrightarrow z f' \in \mathcal{S}^*(0).$
The Koebe one quarter theorem \cite{Duren} ensures that the image of $\triangle$ under
every univalent function $f \in \mathcal{A}$ contains a disk of radius $\frac{1}{4}.$
Thus every univalent function $f$ has an inverse $f^{-1}$ satisfying $$f^{-1}(f(z)) = z,
\ (z \in \triangle)~{\rm and}~~f(f^{-1}(w)) = w\ (|w| < r_{0}(f), \ r_{0}(f)\geq \frac{1}{4}).$$
A function $f \in \mathcal{A}$ is said to be bi-univalent in $\triangle$ if both
$f$ and $f^{-1}$ are univalent in $\triangle.$ Let $\Sigma$ denote the class of
bi-univalent functions defined in the unit disk $\triangle.$ Since $f \in \Sigma$
has the Maclaurian series given by \eqref{c7e1}, a computation shows that its inverse
$g = f^{-1}$ has the expansion
\begin{equation}\label{c7e2a}
g(w) = f^{-1}(w) = w - a_{2} w^2 + (2a_{2}^2 - a_{3})w^3 + \cdots.
\end{equation}
\par An analytic function $f$ is subordinate to an analytic function $g,$ written $f(z) \prec g(z),$
provided there is an analytic function $w$ defined on $\triangle$ with $w(0) = 0$ and
$|w(z)|< 1 $ satisfying $f(z) = g(w(z)).$
\par Chebyshev polynomials, which is used by us in this paper, play a considerable act in numerical analysis. We know that the Chebyshev polynomials are four kinds. The most of books and research articles related to specific orthogonal polynomials of Chebyshev family, contain essentially results of Chebyshev polynomials of first and second kinds $T_{n}(x)$ and $U_{n}(x)$ and their numerous uses in different applications, see Doha \cite{Doha} and Mason \cite{Mason}.

\bigskip

The well-known kinds of the Chebyshev polynomials are the first and second kinds. In the case of real variable $x$ on $(-1, 1)$, the first and second kinds are defined by $$T_{n}(x)=\cos n\theta,$$ $$ U_{n}(x)=\frac{\sin(n+1)\theta}{\sin \theta}$$ where the subcript $n$ denotes the polynomial degree and where $x=\cos \theta.$ We consider the function $$\Phi(z,t)=\frac{1}{1-2tz+z^{2}}.$$
We note that if $t=\cos\alpha$, $\alpha\in\left(\frac{-\pi}{3}, \frac{\pi}{3}\right)$, then for all $z\in \triangle$
\begin{equation*}
\Phi(z,t)=\frac{1}{1-2tz+z^{2}}=1+\sum_{n=1}^{\infty}\frac{\sin(n+1)\alpha}{\sin\alpha}z^{n}=1+2\cos\alpha z+(3\cos^{2}\alpha-\sin^{2}\alpha)z^{2}+\cdots.
\end{equation*}
Thus, we write
\begin{equation*}
\Phi(z,t)=1+U_{1}(t)z+U_{2}(t)z^{2}+...\qquad\qquad (z\in \triangle, t\in(-1,1))
\end{equation*}
where $U_{n-1}=\frac{\sin(n \arccos t)}{\sqrt{1-t^{2}}}$ for $ n\in \mathbb{N},$ are the second kind of the Chebyshev polynomials. Also, it is known that
\begin{equation*}
U_{n}(t)=2tU_{n-1}(t)-U_{n-2}(t),
\end{equation*}
and
\begin{equation}\label{(1.5)}
U_{1}(t)=2t; \qquad\qquad\\
U_{2}(t)=4t^{2}-1,\qquad\qquad\\
U_{3}(t)=8t^{3}-4t,\cdots.
\end{equation}
The Chebyshev polynomials $T_{n}(t), t\in [-1,1],$ of the first kind have the generating function of the form
\begin{equation*}
\sum_{n=0}^{\infty}T_{n}(t)z^{n}=\frac{1-tz}{1-2tz+z^{2}}\qquad\qquad (z\in \triangle).
\end{equation*}

All the same, the Chebyshev polynomials of the first kind $T_{n}(t)$ and the second kind $U_{n}(t)$ are well connected by the following relationship

\begin{equation*}
\frac{dT_{n}(t)}{dt}=nU_{n-1}(t),
\end{equation*}
\begin{equation*}
T_{n}(t)=U_{n}(t)-tU_{n-1}(t),
\end{equation*}
\begin{equation*}
2T_{n}(t)=U_{n}(t)-U_{n-2}(t).
\end{equation*}

\bigskip
\par Several authors have  introduced and investigated subclasses of bi-univalent
functions  and obtained bounds for the initial coefficients (see \cite{bran-clu,bran-saha,lewin,srivastava,xu1,xi}). Recently, Jahangiri and Hamidi \cite{jay} introduced and studied certain subclasses of bi-prestarlike functions mentioned as below:
\par The expansion of $s(z)=\dfrac{z}{{{{\left( {1 - z} \right)}^{2\left( {1 - \alpha } \right)}}}}$ is given by
 \begin{eqnarray*}
s(z) = z + \frac{{\left( {2 - 2\alpha } \right)}}{{1!}}{z^2} + \frac{{\left( {2 - 2\alpha } \right)\left( {3 - 2\alpha } \right)}}{{2!}}{z^3} + \frac{{\left( {2 - 2\alpha } \right)\left( {3 - 2\alpha } \right)\left( {4 - 2\alpha } \right)}}{{3!}}{z^4} +\cdots .
\end{eqnarray*}
So, by the definition of hadamard product, we have
\begin{eqnarray*}\label{c7e3.2a}
F\left( z \right) = \frac{z}{{{{\left( {1 - z} \right)}^{2\left( {1 - \alpha } \right)}}}} * f\left( z \right)=s(z)\ast f(z)
\end{eqnarray*}

 $$F\left( z \right)= z + \frac{{\left( {2 - 2\alpha } \right){a_2}}}{{1!}}{z^2} + \frac{{\left( {2 - 2\alpha } \right)\left( {3 - 2\alpha } \right){a_3}}}{{2!}}{z^3} + \frac{{\left( {2 - 2\alpha } \right)\left( {3 - 2\alpha } \right)\left( {4 - 2\alpha } \right){a_4}}}{{3!}}{z^4} +\cdots$$

equivalently
\begin{equation}\label{e3.2}
F\left( z \right)= z + \Psi_2(\alpha)a_2z^2 + \Psi_3(\alpha)a_3z^3 + \Psi_4(\alpha)a_4z^4 +\cdots
\end{equation}
Similarly, for the inverse function $g = {f^{ - 1}}$, we obtain
\begin{equation*}
G\left( w \right) = \frac{w}{{{{\left( {1 - w} \right)}^{2\left( {1 - \alpha } \right)}}}} * g\left( w \right)=s(w)\ast g(w)
\end{equation*}

$$G\left( w \right) = w - \frac{{\left( {2 - 2\alpha } \right){a_2}}}{{1!}}{w^2} + \frac{{\left( {4{{\left( {2 - 2\alpha } \right)}^2}a_2^2 - \left( {2 - 2\alpha } \right)\left( {3 - 2\alpha } \right){a_3}} \right)}}{{2!}}{w^3}+\cdots $$

equivalently
\begin{equation}\label{c7e2}
G\left( w \right) = w - \Psi_2(\alpha)a_2{w^2} + \left( 2\Psi^2_2(\alpha)a_2^2 - \Psi_3(\alpha)a_3 \right){w^3}+\cdots
\end{equation}

We define bi-prestarlike functions  in the open unit disk, associated with Chebyshev Polynomials as below:
\begin{definition}\label{c7def3.1}
For $0\leq \lambda \leq 1,$ and $t\in(0,1)$ a  function $f \in \Sigma$  of the form (\ref{c7e1}) is said to be in the class $\mathcal{R}_{\Sigma}(\lambda,\alpha, \Phi(z,t))$
if the following subordination hold:
\begin{equation}\label{c7e3.1}
(1 - \lambda) \frac{zF'(z)} { F(z)}+ \lambda \left(1+\frac{zF''(z)} { F'(z)}\right) \prec \Phi(z,t)
\end{equation}
and
\begin{equation}\label{c7e3.2}
(1 - \lambda) \frac{wG'(w)} { G(w)}+ \lambda \left(1+\frac{w G''(w)} { G'(w)}\right)  \prec \Phi(w,t)
\quad
\end{equation}where $z,w \in \Delta$ and $F$ and $G$ is given by \eqref{c7e3.2a} and (\ref {c7e2}), respectively.
\end{definition}

\begin{remark}\label{c7rem3.2}
Suppose $f\in\Sigma.$ Then $\mathcal{R}_{\Sigma}(0,\alpha,\Phi(z,t))\equiv\mathcal{PS}_{\Sigma}^{*}(\alpha,\Phi(z,t)):$thus $f\in \mathcal{PS}_{\Sigma}^{*}(\alpha,\Phi(z,t))$ if the following subordination holds:
\begin{equation*}
 \frac{zF'(z)} { F(z)} \prec \Phi(z,t) \quad and \quad \frac{wG'(w)} { G(w)}  \prec \Phi(w,t)
\end{equation*}
where $z,w \in \Delta$ and $G$ is given by (\ref {c7e2}).
\end{remark}

\begin{remark}\label{c7rem3.3}
Suppose $f\in\Sigma.$ Then $\mathcal{R}_{\Sigma}(1,\alpha, \Phi(z,t))\equiv\mathcal{K}_{\Sigma}^{*}(\alpha,\Phi(z,t)):$ thus
 $f\in \mathcal{K}_{\Sigma}^{*}(\alpha,\Phi(z,t))$ if the following subordination holds:
\begin{equation*}
1+\frac{zF''(z)} { F'(z)} \prec \Phi(z,t)\quad and \quad 1+\frac{w G''(w)} { G'(w)} \prec \Phi(w,t)
\end{equation*}
where $z,w \in \Delta$ and $g$ is given by (\ref {c7e2}).
\end{remark}

 In this paper, motivated by recent works of Alt{\i}nkaya and
Yal\c{c}{\i}n \cite{01} we introduce a subclass bi-prestarlike function class associated with Chebyshev polynomials and obtain the initial Taylor coefficients
$|a_{2}|$ and $|a_{3}|$ for the functions $f\in\mathcal{R}_{\Sigma}(\lambda,\alpha, \Phi(z,t))$  by subordination.

\section {Initial Taylor Coefficients  $f\in\mathcal{R}_{\Sigma}(\lambda,\alpha, \Phi(z,t))$ }
\begin{theorem}\label{c7thm3.4}
Let $f$ given by \eqref{c7e1} be in the class $\mathcal{R}_{\Sigma}(\lambda,\alpha, \Phi(z,t))$ and $t\in (0,1).$ Then
\begin{equation}\label{c7e3.3}
|a_{2}| \leq \frac{ 2t \sqrt{2t}} {\sqrt{|[2(1+2\lambda)\Psi_3(\alpha)-(\lambda^2 + 5\lambda+2)\Psi^2_2(\alpha)]4t^2
+ (1 + \lambda)^2 \Psi^2_2(\alpha) |} }
\end{equation}
and
\begin{equation}\label{c7e3.4}
|a_{3}| \leq  \frac{4t^2}{(1+\lambda)^2 \Psi^2_2(\alpha)}+\frac{t}{(1+2 \lambda) \Psi_3(\alpha)}
\end{equation}where $0\leq \lambda \leq 1$ and $t\neq\frac{(1 + \lambda) \Psi_2(\alpha)}{\sqrt{(\lambda^2 + 5\lambda+2)\Psi^2_2(\alpha)-2(1+2\lambda)\Psi_3(\alpha)}}.$
\end{theorem}

\begin{proof}
Let $f \in \mathcal{R}_{\Sigma}(\lambda,\alpha, \Phi(z,t))$ and $g = f^{-1}.$ Considering (\ref{c7e3.1}) and (\ref{c7e3.2}), we have
\begin{equation}\label{c7e2.3}
(1 - \lambda) \frac{zF'(z)} { F(z)}+ \lambda \left(1+\frac{zF''(z)} { F'(z)}\right) = \Phi(z,t)
\end{equation}
and
\begin{equation}\label{c7e2.4}
(1 - \lambda) \frac{wG'(w)} { G(w)}+ \lambda \left(1+\frac{w G''(w)} { G'(w)}\right) = \Phi(w,t).
\end{equation}
Define the functions $u(z)$ and $v(w)$ by
\begin{equation}\label{c7e2.5}
u(z)  =  c_{1} z + c_{2} z^2 + \cdots
\end{equation}
and
\begin{equation}\label{c7e2.6}
v(w) = d_{1} w + d_{2} w^2 + \cdots
\end{equation}
are analytic in $\triangle$ with $u(0) = 0 = v(0)$
and $|u(z)|<1,$ $|v(w)|<1,$ for all $z, w\in \triangle.$ It is well-known that \begin{equation}\label{(u)}|u(z)|=|c_{1}z+c_{2}z^{2}+\cdots|<1 \qquad {\rm  and }\qquad  |v(w)|=|d_{1}w+d_{2}w^{2}+\cdots|<1,z,w\in\triangle,\end{equation}

then \begin{equation}\label{(c)}|c_{j}|\leq1\qquad {\rm  and }\qquad |d_{j}|\leq 1 \quad{\rm for~~ all }\quad j\in \mathbb{N}.\end{equation}
\par Using \eqref{c7e2.5} and \eqref{c7e2.6} in \eqref{c7e2.3} and \eqref{c7e2.4} respectively, we have
\begin{equation}\label{c7e2.7}
(1 - \lambda) \frac{zF'(z)} { F(z)}+ \lambda \left(1+\frac{zF''(z)} { F'(z)}\right)= 1+U_{1}(t)u(z)+U_{2}(t)u^{2}(z)+\cdots,
\end{equation}
and
\begin{equation}\label{c7e2.8}
(1 - \lambda) \frac{wG'(w)} { G(w)}+ \lambda \left(1+\frac{w G''(w)} { G'(w)}\right) = 1+U_{1}(t)v(w)+U_{2}(t)v^{2}(w)+\cdots.
\end{equation}
In light of \eqref{c7e1} - \eqref{c7e2a}, and  from \eqref{c7e2.7} and \eqref{c7e2.8},
we have
\begin{multline*}
1 + (1 + \lambda)\Psi_2(\alpha) a_{2} z + [2 (1 + 2 \lambda )\Psi_3(\alpha)a_{3} - (1 +3\lambda)\Psi^{2}_2(\alpha)a_{2}^2] z^2 + \cdots\\
= 1+U_{1}(t)c_{1}z+[U_{1}(t)c_{2}+U_{2}(t)c_{1}^{2}]z^{2}+\cdots,
\qquad \qquad \qquad
\end{multline*}
and
\begin{multline*}
1 - (1 + \lambda)\Psi_2(\alpha)a_{2} w + \{[(8\lambda+4)\Psi_3(\alpha) - (3\lambda +1)\Psi^2_2(\alpha))] a_{2}^2-2 (1 + 2 \lambda) \Psi_3(\alpha)a_{3}\} w^2 + \cdots\\
= 1+U_{1}(t)d_{1}w+[U_{1}(t)d_{2}+U_{2}(t)d_{1}^{2}]w^{2}+\cdots.
\qquad \qquad \qquad
\end{multline*}
which yields the following relations:

\begin{eqnarray}
 (1 + \lambda) \Psi_2(\alpha) a_{2}  &=&  U_{1}(t)c_{1},\label{c7e3.5} \\
 -(1 + 3\lambda) \Psi^2_2(\alpha) a_{2}^2 + 2(1 + 2 \lambda) \Psi_3(\alpha) a_{3} &=& U_{1}(t)c_{2}+U_{2}(t)c_{1}^{2} \label{c7e3.6}
\end{eqnarray}
and
\begin{eqnarray}
  - (1 + \lambda) \Psi_2(\alpha) a_{2}   &=&  U_{1}(t)d_{1}, \qquad\qquad\qquad \label{c7e3.7}\\
( 4 (1 + 2 \lambda) \Psi_3(\alpha) -(1 + 3\lambda) \Psi^2_2(\alpha) )a_{2}^2 - 2(1 + 2 \lambda) \Psi_3(\alpha) a_{3}  &=& U_{1}(t)d_{2}+U_{2}(t)d_{1}^{2}.\qquad\qquad\qquad\label{c7e3.8}
\end{eqnarray}
From \eqref{c7e3.5} and \eqref{c7e3.7} it follows that
\begin{equation}\label{c7e3.9}
c_{1} = -d_{1}
\end{equation}
and
\begin{equation}\label{c7e3.10}
2(1 + \lambda)^2 \Psi^2_2(\alpha) a_{2}^2 =U_{1}^2(t) (c_{1}^2 + d_{1}^2).
\end{equation}
Adding \eqref{c7e3.6} to \eqref{c7e3.8} and using  \eqref{c7e3.10}, we obtain
\begin{equation*}
a_{2}^2 = \frac{U_{1}^3 (t)(c_{2} + d_{2}) }{2 [\{ 2(1 + 2 \lambda) \Psi_3(\alpha) - (1 + 3\lambda) \Psi^2_2(\alpha)\}U_{1}^2(t)
-(1 + \lambda)^2\Psi^2_2(\alpha)  U_{2}(t)]}.
\end{equation*}
Applying  (\ref{(c)}) to the coefficients $c_{2}$ and $d_{2},$  and using (\ref{(1.5)})we have
\begin{equation}
|a_{2}| \leq \frac{ 2t \sqrt{2t} } { \sqrt{|[2(1+2\lambda) \Psi_3(\alpha) - (\lambda^2 + 5\lambda+2)\Psi^2_2(\alpha)]4t^2
+ (1 + \lambda)^2\Psi^2_2(\alpha) |} }.
\end{equation}
By subtracting \eqref{c7e3.8} from \eqref{c7e3.6} and using \eqref{c7e3.9} and \eqref{c7e3.10}, we get
\begin{equation*}
a_{3} = \frac{U_{1}^2(t) (c_{1}^2 + d_{1}^2)}{2 (1 + \lambda)^2\Psi^2_2(\alpha)} + \frac{U_{1} (c_{2} - d_{2})}{4 (1 + 2 \lambda) \Psi_3(\alpha)}.
\end{equation*}
Using (\ref{(1.5)}), once again applying  (\ref{(c)}) to the coefficients $c_{1}, c_{2}, d_{1}$ and $d_{2},$ we get
\begin{equation}
|a_{3}| \leq  \frac{4t^2}{(1 + \lambda)^2 \Psi^2_2(\alpha)}+\frac{t}{(1 + 2 \lambda) \Psi_3(\alpha)}.
\end{equation}
\end{proof}
\par By taking $\lambda=0~ {\rm or }~\lambda=1$ and $ t\in (0,1),$ one can easily state the  estimates $|a_2| ~{\rm and}~ |a_3|$ for the function classes $\mathcal{R}_{\Sigma}(0,\alpha,\Phi(z,t)) = \mathcal{PS}_{\Sigma}^{*}(\alpha,\Phi(z,t))  $ and $ \mathcal{R}_{\Sigma}(1,\alpha, \Phi(z,t))
= \mathcal{K}^{*}_{\Sigma}(\alpha,\Phi(z,t))$ respectively.
\begin{remark}\label{c7rem3.6a}
Let $f$ given by \eqref{c7e1} be in the class $\mathcal{PS}_{\Sigma}^{*}(\alpha,\Phi(z,t)).$ Then
\begin{equation*}
|a_{2}| \leq \frac{ 2t \sqrt{2t} } { \sqrt{|[\Psi_3(\alpha) -\Psi^2_2(\alpha)]8t^2
+\Psi^2_2(\alpha) |} }
\end{equation*}
and
\begin{equation*}
|a_{3}| \leq  \frac{4t^2}{\Psi^2_2(\alpha)}+\frac{t}{\Psi_3(\alpha)}.
\end{equation*}
where $t\neq\frac{\Psi_2(\alpha)}{2\sqrt{2\Psi^2_2(\alpha)-2\Psi_3(\alpha)}}.$
\end{remark}

\begin{remark}
Let $f$ given by \eqref{c7e1} be in the class $\mathcal{K}_{\Sigma}^{*}(\alpha,\Phi(z,t)).$ Then
\begin{equation}\label{c7e3.3}
|a_{2}| \leq \frac{ 2t \sqrt{2t} } { \sqrt{|[3\Psi_3(\alpha) - 4\Psi^2_2(\alpha)]8t^2
+  4\Psi^2_2(\alpha) |} }
\end{equation}
and
\begin{equation}\label{c7e3.4}
|a_{3}| \leq  \frac{t^2}{\Psi^2_2(\alpha)}+\frac{t}{3\Psi_3(\alpha)}.
\end{equation}
where  $t\neq\frac{\Psi_2(\alpha)}{\sqrt{8\Psi^2_2(\alpha)-6\Psi_3(\alpha)}}.$
\end{remark}
For $\alpha =0,$ Theorem \ref{c7thm3.4} yields the following corollary.
\begin{corollary}\label{c1}
Let $f$ given by \eqref{c7e1} be in the class $\mathcal{R}_{\Sigma}(\lambda,0, \Phi(z,t)).$ Then
\begin{equation*}
|a_{2}| \leq \frac{ t \sqrt{2t} } { \sqrt{| (1 + \lambda)^2-2 (2\lambda^2 +4 \lambda+1)t^2
|} }
\end{equation*}
and
\begin{equation*}
|a_{3}| \leq  \frac{t^2}{(1 + \lambda)^2 }+\frac{t}{3(1 + 2 \lambda) }
\end{equation*}where $0\leq \lambda \leq 1$ and $t\neq\frac{\sqrt{1+\lambda}}{2\sqrt{\lambda}}$ for $\lambda\neq0$
\end{corollary}

 \par By taking $\alpha=0$ in the above remarks  we get the  estimates $|a_2| ~{\rm and}~ |a_3|$ for the function classes $ \mathcal{S}_{\Sigma}^{*}(\frac{1}{2},\Phi(z,t)) $ and $ \mathcal{K}^*_{\Sigma}(\frac{1}{2},\Phi(z,t)).$
 \begin{remark}
Let $f$ given by \eqref{c7e1} be in the class $\mathcal{S}_{\Sigma}^{*}(\frac{1}{2},\Phi(z,t)).$ Then
\begin{equation*}
|a_{2}| \leq 2t \sqrt{2t}
\end{equation*}
and
\begin{equation*}
|a_{3}| \leq  {4t^2}+t.
\end{equation*}
\end{remark}

\begin{remark}
Let $f$ given by \eqref{c7e1} be in the class $\mathcal{K}_{\Sigma}^{*}(\frac{1}{2},\Phi(z,t)).$ Then for $t\neq\frac{1}{\sqrt{2}}$,
\begin{equation*}
|a_{2}| \leq \frac{ 2t \sqrt{2t} } { \sqrt{|4-8t^2|} }
\end{equation*}
and
\begin{equation*}
|a_{3}| \leq  {t^2}+\frac{t}{3}.
\end{equation*}
\end{remark}

\section{Fekete-Szeg\"{o} inequality for the function class $\mathcal{R}_{\Sigma}(\lambda,\alpha, \Phi(z,t))$ }
Due to Zaprawa \cite {zap}, in this section we obtain the Fekete-Szeg\"{o} inequality for the function classes $\mathcal{R}_{\Sigma}(\lambda,\alpha, \Phi(z,t)).$
\begin{theorem}\label{thm2} Let $f$ given by \eqref{c7e1} be in the class $\mathcal{R}_{\Sigma}(\lambda,\alpha, \Phi(z,t))$ and $\mu\in\mathbb{R}$. Then we have
\begin{multline*}
|a_{3}-\mu a_{2}^{2}|\leq\\
\left\{ \begin{array}{ll}
\frac{t}{(1+2\lambda)\Psi_3(\alpha)},\,\,\,\,\,\,\,\,\,\,\,\,\,\,\,\,\,\,\,\,\,\,\,\,\,\,\,\,\,\,\,\,\,\,\,\,\,\,\,\,\,\,\,\,\,\,\,\,\,\,\,\,\,\,\,\,\,\,\,\,\,\,\,\,\,\,\,\,\,\,\,\,\,\,\,\,\,\,\,\,\, \mbox{ $|\mu-1| \leq \frac{\left|\frac{(1+\lambda)^{2}\Psi^2_2(\alpha)}{4t^{2}}+2(1+2\lambda)\Psi_3(\alpha)-(\lambda^{2}+5\lambda+2)\Psi^2_2(\alpha)\right|}{2(1+2\lambda)\Psi_3(\alpha)}$}\\
\\
\frac{8|1-\mu|t^{3}}{|(2(1+2\lambda)\Psi_3(\alpha)-(\lambda^{2}+5\lambda+2)\Psi^2_2(\alpha))4t^{2}+(1+\lambda)^{2}\Psi^2_2(\alpha)|} ,\,\,\, \mbox{ $|\mu-1| \geq \frac{\left|\frac{(1+\lambda)^{2}\Psi^2_2(\alpha)}{4t^{2}}+2(1+2\lambda)\Psi_3(\alpha)-(\lambda^{2}+5\lambda+2)\Psi^2_2(\alpha)\right|}{2(1+2\lambda)\Psi_3(\alpha)}$}\end{array}\right. \end{multline*}
\end{theorem}
\textbf{Proof.} From (\eqref{c7e3.6}) and (\eqref{c7e3.8})
\begin{multline}\label{(3.1)}
a_{3}-\mu a_{2}^{2}=\nonumber\\(1-\mu)\frac{U_{1}^{3}(t)(c_{2}+d_{2})}{(4(1+2\lambda)
\Psi_3(\alpha)-2(1+3\lambda)\Psi^2_2(\alpha))U_{1}^{2}(t)-2U_{2}(t)(1+\lambda)^{2}\Psi^2_2(\alpha)}\nonumber
\\+\frac{U_{1}(t)(c_{2}-d_{2})}{4(1+2\lambda)\Psi_3(\alpha)}
\end{multline}
\begin{equation}\label{(3.1)}
=U_{1}(t)\left[\left(h(\mu)+\frac{1}{4(1+2\lambda)\Psi_3(\alpha)}\right)c_{2}+\left(h(\mu)-\frac{1}{4(1+2\lambda)\Psi_3(\alpha)}\right)d_{2}\right]
\end{equation}
where
\begin{equation}\label{(3.1)}
h(\mu)=\frac{(1-\mu)U_{1}^{2}(t)}{2[2(1+2\lambda)\Psi_3(\alpha)-(1+3\lambda)\Psi^2_2(\alpha)) U_{1}^{2}(t)-(1+\lambda)^{2}\Psi^2_2(\alpha)U_{2}(t)]}.
\end{equation}

Then, in view of (\ref{(1.5)}), we conclude that

\[|a_{3}-\mu a_{2}^{2}|\leq\left\{ \begin{array}{ll}
\frac{t}{(1+2\lambda)\Psi_3(\alpha)} , & \mbox{ $0\leq |h(\mu)| \leq \frac{1}{4(1+2\lambda)\Psi_3(\alpha)}$}\\
4t|h(\mu)| , & \mbox{ $ |h(\mu)|\geq \frac{1}{4(1+2\lambda)\Psi_3(\alpha)}$}\end{array} \right. \]

Taking $\mu=1$, we have the following corollary.
\begin{corollary}\label{c3} If $f\in \mathcal{R}_{\Sigma}(\lambda,\alpha, \Phi(z,t))$ , then
\begin{equation}\label{(3.1)}
|a_{3}-a_{2}^{2}|\leq\frac{t}{(1+2\lambda)\Psi_3(\alpha)}.
\end{equation}
\end{corollary}

\begin{corollary}\label{c4} Let $f$ given by \eqref{c7e1} be in the class $\mathcal{S}_{\Sigma}^{*}(\alpha,\Phi(z,t))$ and $\mu\in\mathbb{R}$. Then we have
\[|a_{3}-\mu a_{2}^{2}|\leq\left\{ \begin{array}{ll}
\frac{t}{\Psi_3(\alpha)} , & \mbox{ $|\mu-1| \leq \frac{\left|\frac{\Psi^2_2(\alpha)}{8t^{2}}+\Psi_3(\alpha)-\Psi^2_2(\alpha)\right|}{\Psi_3(\alpha)}$}\\
\frac{8|1-\mu|t^{3}}{|((\Psi_3(\alpha)-\Psi^2_2(\alpha))8t^{2}+\Psi^2_2(\alpha)|} , & \mbox{ $|\mu-1| \geq \frac{\left|\frac{\Psi^2_2(\alpha)}{8t^{2}}+\Psi_3(\alpha)-\Psi^2_2(\alpha)\right|}{\Psi_3(\alpha)}$}.\end{array}\right.\]
Especially, for $\mu=1$ if $f\in \mathcal{S}_{\Sigma}^{*}(\frac{1}{2},\Phi(z,t))$ we obtain
\begin{equation}\label{(3.1)}
|a_{3}-a_{2}^{2}|\leq\ t.
\end{equation}
\end{corollary}

\begin{corollary}\label{c5} Let $f$ given by \eqref{c7e1}
be in the class $\mathcal{K}_{\Sigma}^{*}(\alpha,\Phi(z,t))$ and $\mu\in\mathbb{R}$. Then we have
\[|a_{3}-\mu a_{2}^{2}|\leq\left\{ \begin{array}{ll}
\frac{t}{3\Psi_3(\alpha)} , & \mbox{ $|\mu-1| \leq \frac{\left|\frac{\Psi^2_2(\alpha)}{2t^{2}}+3\Psi_3(\alpha)-4\Psi^2_2(\alpha)\right|}{3\Psi_3(\alpha)}$}\\
\frac{2|1-\mu|t^{3}}{|((3\Psi_3(\alpha)-4\Psi^2_2(\alpha))2t^{2}+\Psi^2_2(\alpha)|} , & \mbox{ $|\mu-1| \geq \frac{\left|\frac{\Psi^2_2(\alpha)}{2t^{2}}+3\Psi_3(\alpha)-4\Psi^2_2(\alpha)\right|}{3\Psi_3(\alpha)}$}.\end{array}\right.\]
Especially, for $\mu=1$ if $f\in \mathcal{K}_{\Sigma}^{*}(\frac{1}{2},\Phi(z,t))$ we obtain
\begin{equation}\label{(3.1)}
|a_{3}-a_{2}^{2}|\leq\ \frac{t}{3}.
\end{equation}
\end{corollary}

\section{Conflicts of Interest}
The authors declare that they have no conflicts of interest regarding the publication of this paper.
\section{Acknowledgement}
We authors record our sincere thanks to the authorities of ICAMS-2017, Department of Mathematics, School of Advanced Sciences, VIT, Vellore-632 014, India for having given an opportunity to present the paper in ICAMS-2017.

\end{document}